%% file: GO_Wilks_problem_2026-07-04__as_submitted_.tex
\documentclass[12pt,authoryear]{elsarticle}
\journal{}

\makeatletter
\def\ps@pprintTitle{%
 \let\@oddhead\@empty
 \let\@evenhead\@empty
 \def\@oddfoot{\hfill\thepage}%
 \let\@evenfoot\@oddfoot}
\makeatother

\usepackage{amsmath,amssymb,amsfonts,amsthm,graphicx}
\usepackage[bookmarksnumbered,colorlinks=true]{hyperref}
\usepackage[labelfont=bf]{caption}

\providecommand{\doi}[1]{\href{https://doi.org/#1}{DOI:#1}}
\usepackage{xurl} 
\renewcommand{\doi}[1]{%
 \href{https://doi.org/#1}{\nolinkurl{DOI:#1}}%
}

\usepackage{dsfont} 
\usepackage{enumitem} 
\usepackage{mathtools} 
\usepackage{geometry} 
\numberwithin{equation}{section}
\numberwithin{table}{section}
\numberwithin{figure}{section}

\theoremstyle{plain}
\newtheorem{theorem}{Theorem}[section]

\theoremstyle{definition}
\newtheorem{remark}{Remark}[section]

\newcommand{\N}{\mathbb{N}}
\newcommand{\R}{\mathbb{R}}
\newcommand{\EE}{\mathsf{E}} 
\newcommand{\bb}[1]{\boldsymbol{#1}}
\newcommand{\rd}{\mathrm{d}}
\newcommand{\e}{\varepsilon}
\newcommand{\tr}{\mathrm{tr}}
\newcommand{\etr}{\mathrm{etr}}
\newcommand{\leqdef}{\vcentcolon=}

\allowdisplaybreaks

\begin{document}

\begin{frontmatter}

\title{On Wilks' problem: Exact recursive formulas via Stein's method for the joint moments of disjoint principal minors of Wishart random matrices}

\author[a2]{Robert E.\ Gaunt}
\author[a3]{Fr\'ed\'eric Ouimet\corref{mycorrespondingauthor}}

\address[a2]{The University of Manchester, Manchester, M13 9PL, UK}
\address[a3]{Universit\'e du Qu\'ebec \`a Trois-Rivi\`eres, Trois-Rivi\`eres, QC G8Z 4M3, Canada\vspace{-5mm}}

\cortext[mycorrespondingauthor]{Corresponding author. Email address: frederic.ouimet2@uqtr.ca}

\begin{abstract}
In a 1934 {\it Annals of Mathematics} paper, Samuel S.\ Wilks posed the problem of computing joint moments of disjoint principal minors of Wishart random matrices, describing the general case as ``extremely complicated.'' These moments arise in classical multivariate statistics, including multivariate regression, analysis of variance, generalized correlation coefficients, and likelihood ratio testing. We solve Wilks' problem for nonnegative integer exponents and any finite number of disjoint principal minors using Stein's method, specifically the Wishart Stein characterization recently developed by \citet{BaillyGauntOuimetRichardsvonSachs2026}. The Wishart Stein identity yields a degree-lowering recursion that terminates after finitely many steps and evaluates the desired moment exactly. Numerical experiments validate our recursive formulas against direct Monte Carlo simulations from the Wishart distribution.
\end{abstract}


\end{frontmatter}

\section{Introduction}\label{sec:intro}

The Wishart distribution was introduced by \citet{doi:10.2307/2331939} in his study of empirical covariance matrices from normal multivariate populations. It is the matrix-variate analogue of the gamma distribution and remains one of the most important laws of multivariate statistics, with applications in covariance estimation, multivariate linear models, MANOVA, likelihood ratio testing, Bayesian inference for Gaussian precision matrices, and random matrix theory; see, e.g., \citet{MR652932} and \citet{MR1990662}.

Let $\mathcal{S}^p$, $\mathcal{S}_{+}^p$, and $\mathcal{S}_{++}^p$ denote the sets of real symmetric, nonnegative definite, and positive definite $p\times p$ matrices, respectively. For any square matrix $A$, let $\tr(A)$ be its trace, $\etr(A)\leqdef \exp\{\tr(A)\}$, and $|A|$ its determinant. For $\alpha\in(p-1,\infty)$ and $\Sigma\in\mathcal{S}_{++}^p$, the central Wishart distribution $\mathcal{W}_p(\alpha,\Sigma)$ has density, relative to Lebesgue measure on $\mathcal{S}^p$,
\[
f_{\alpha,\Sigma}(S) = \frac{|S|^{\alpha/2-(p + 1)/2}\etr(-\Sigma^{-1}S/2)}{|2\Sigma|^{\alpha/2}\Gamma_p(\alpha/2)}, \qquad S\in\mathcal{S}_{++}^p,
\]
where $\Gamma_p$ denotes the multivariate gamma function. If a random matrix $\mathfrak{W}$ has this distribution, we write $\mathfrak{W}\sim\mathcal{W}_p(\alpha,\Sigma)$. Its Laplace transform has a well-known explicit expression, namely $\EE\{\etr(-T\mathfrak{W})\}=|I_p + 2T\Sigma|^{-\alpha/2}$ for $T\in\mathcal{S}_{+}^p$. Differentiating this transform gives moments of matrix entries, but it does not provide a tractable procedure for extracting the determinant moments considered in \eqref{eq:intro.Wilks.problem} below. Products of determinant powers are nonlinear functions of many dependent entries, which is the source of the difficulty addressed here.

Let $d,p_1,\ldots,p_d\in\N$, set $p=p_1 + \cdots + p_d$, and partition $\mathfrak{W}=(\mathfrak{W}_{ij})_{1 \leq i,j \leq d}$ and $\Sigma=(\Sigma_{ij})_{1 \leq i,j \leq d}$, where the block $\mathfrak{W}_{ij}$, and similarly $\Sigma_{ij}$, has size $p_i\times p_j$. Thus $\mathfrak{W}_{ii}$ is the $i$-th principal diagonal block and $|\mathfrak{W}_{ii}|$ denotes its determinant. For $\nu_1,\ldots,\nu_d\in[0,\infty)$, Wilks' problem is to evaluate the following joint moments of disjoint principal minors of Wishart random matrices:
\begin{equation}\label{eq:intro.Wilks.problem}
\EE\left[\prod_{i=1}^d |\mathfrak{W}_{ii}|^{\nu_i}\right].
\end{equation}

The question arose from Wilks' work on generalized statistics in multivariate analysis of variance and regression. In his earlier paper \citep{doi:10.2307/2331979} and his 1934 {\it Annals of Mathematics} paper \citep{MR1503165}, Wilks developed moment-generating operators for determinants of product moments in samples from a normal system. The moments in \eqref{eq:intro.Wilks.problem} enter the exact study of generalized likelihood ratio statistics, generalized correlation coefficients, and multivariate analogues of Student-type statistics, and are therefore tied directly to classical sampling theory for multivariate regression, analysis of variance, and tests of independence among several sets of variables.

A closely related but more tractable problem concerns embedded principal minors. For $i\in \{1,\ldots,d\}$, let $\mathfrak{W}_{1:i,1:i}\leqdef (\mathfrak{W}_{rs})_{1 \leq r,s \leq i}$ be the leading principal submatrix formed from the first $i$ block rows and columns, of size $(p_1 + \cdots + p_i)\times(p_1 + \cdots + p_i)$. The embedded problem asks for $\smash{\EE[\prod_{i=1}^d |\mathfrak{W}_{1:i,1:i}|^{\nu_i}]}$. Wilks derived these moments for integer-valued shape parameters using moment-generating operators that commute with integration, Cauchy's expansion for the determinant of a bordered matrix, and Jacobi's identity for complementary minors. His lengthy and technically sophisticated argument predates the modern use of the multivariate gamma function and the theory of hypergeometric functions of matrix argument developed by \citet{MR69960}. \citet{MR181056} later extended the formulas to every real $\alpha>p-1$ using Bartlett's decomposition and zonal polynomials, while \citet{MR4757806} recently gave a streamlined third proof based on iterated Schur complements. The latter proof exploits the nested structure of the leading principal submatrices: each successive determinant is expressed as the determinant of the preceding leading block times the determinant of an independent Wishart Schur complement.

The diagonal blocks in \eqref{eq:intro.Wilks.problem} are independent when $\Sigma$ is block diagonal, in which case the joint moment factors into marginal principal-minor moments. For a general scale matrix, however, the off-diagonal blocks of $\Sigma$ create dependence among the diagonal blocks of $\mathfrak{W}$. Because these diagonal blocks are disjoint rather than nested, applying a Schur complement to one block alters the remaining factors instead of reducing the problem to determinants of independent Wishart matrices. Several related classes of Wishart moments are nevertheless known. A single principal-minor moment formula follows from the marginal Wishart law \citep[see, e.g.,][Theorem~3.3.22]{GuptaNagar2000}; \citet{MR1868979} studied moments of multivariate monomials in the entries of a Wishart matrix; \citet{MR2066255} treated invariant polynomial moments; \citet{MR2458187} derived first- and second-order moments of arbitrary minors; and \citet{doi:10.1007/s40953-021-00267-7,doi:10.1111/sjos.12707} considered expectations of equivariant matrix-valued functions of Wishart and inverse Wishart matrices.

A first major step toward evaluating \eqref{eq:intro.Wilks.problem} was made by \citet{MR4862164}, who solved the case $d=2$ through an explicit formula involving the Gaussian hypergeometric function of matrix argument, ${}_2F_1$. Their result extends the classical bivariate normal absolute-moment formula of \citet{MR45347}. Beyond two blocks, no analogous representation in terms of known special functions is available. This mirrors the Gaussian setting: explicit special-function formulas for joint absolute moments are classical in dimension two, whereas no such formula is known in dimension three or above; the best general results are series expansions of the integral representations defining the moments, which must be truncated for numerical evaluation; see \citet{MR52072} and \citet{MR4219331}. By contrast, the recursive formulas developed here for nonnegative integer powers are exact, finite, and numerically computable without series truncation.

Exact evaluations of \eqref{eq:intro.Wilks.problem} are useful beyond historical completeness. They provide cumulants and benchmarks for approximations to likelihood ratio statistics, help quantify generalized correlation measures, and furnish exact test cases for asymptotic theory in multivariate analysis.

Our contribution is to give the first general exact recursive formulas for \eqref{eq:intro.Wilks.problem} when the powers are nonnegative integers. The proof uses Stein's method, specifically the Wishart Stein characterization recently developed by \citet{BaillyGauntOuimetRichardsvonSachs2026}. Applying that identity to homogeneous determinant polynomials yields a degree-lowering recursion that terminates after finitely many steps. The resulting algorithm applies to any finite number of blocks, arbitrary positive block sizes, every $\alpha>p-1$, and every $\Sigma\in\mathcal{S}_{++}^p$.

The remainder of the paper is organized as follows. Section~\ref{sec:definitions} introduces further notation for block matrices, symmetric gradients, and homogeneous polynomials. Section~\ref{sec:result} states and proves the recursive formula. Section~\ref{sec:validation} implements the recursion and validates it numerically against Monte Carlo simulation.

\section{Definitions and notation}\label{sec:definitions}

Let $p_1,\ldots,p_d\in\N$, and set $p\leqdef p_1 + \cdots + p_d$. Whenever $S\in \mathcal{S}^p$ is block-partitioned according to these sizes, we write $S=(S_{rs})_{1 \leq r,s \leq d}$, where $S_{rs}\in \R^{p_r\times p_s}$. Thus $S_{rr}$ is the $r$-th principal diagonal block. Any finite family of pairwise disjoint principal minors can be put in this form after a simultaneous permutation of rows and columns; if the selected index sets do not cover $\{1,\ldots,p\}$, the complement can be added as an extra block with exponent zero. For $\bb{n}=(n_1,\ldots,n_d)\in\N_0^d$, define
\[
P_{\bb{n}}(S)\leqdef \prod_{r=1}^d |S_{rr}|^{n_r}, \qquad S\in \mathcal{S}^p,
\]
and $m_{\bb{n}}\leqdef \sum_{r=1}^d n_r p_r$. For $m\in\N_0$, let $\mathcal{H}_m$ denote the vector space of real homogeneous polynomials of total degree $m$ in the independent entries of a symmetric matrix. By convention, $\mathcal{H}_0$ is the space of constant polynomials.

For a differentiable function $f:\mathcal{S}^p\to\R$, we use the symmetric gradient
\[
\nabla f(S) \leqdef \left(\frac{1}{2}(1 + \delta_{ij})\frac{\partial f}{\partial S_{ij}}(S)\right)_{1\leq i,j\leq p},
\]
where $\delta_{ij}$ denotes the Kronecker delta. Under this convention, each pair $S_{ij}=S_{ji}$, $i\neq j$, represents a single coordinate, so $\partial/\partial S_{ij}=\partial/\partial S_{ji}$ for off-diagonal entries. Thus $\nabla_{ii}=\partial/\partial S_{ii}$, while $\nabla_{ij}=\frac{1}{2}\partial/\partial S_{ij}$ for $i\neq j$.

For a twice differentiable function $f:\mathcal{S}^p\to\R$, we use the notation
\[
\tr\{S\nabla\Sigma\nabla f(S)\} \leqdef \sum_{i,j,k,\ell=1}^p S_{ik}\Sigma_{\ell j}\nabla_{ij}\nabla_{k\ell}f(S), \qquad S\in\mathcal{S}^p.
\]

\section{Result}\label{sec:result}

Let $\alpha\in (p-1,\infty)$ and $\Sigma\in \mathcal{S}_{++}^p$. The Wishart Stein identity used below is a consequence of the forward implication in Corollary~3.3 of \citet{BaillyGauntOuimetRichardsvonSachs2026} and is formulated as follows: if $\mathfrak{W}\sim \mathcal{W}_p(\alpha,\Sigma)$ and $f$ is a real-valued polynomial function on $\mathcal{S}^p$, then
\begin{equation}\label{eq:Wishart.Stein.identity.polynomial}
\EE\left[2 \, \tr\{(\alpha\Sigma-\mathfrak{W})\nabla f(\mathfrak{W})\}  +  4 \, \tr\{\mathfrak{W}\nabla\Sigma\nabla f(\mathfrak{W})\}\right] = 0.
\end{equation}
The restriction to polynomials is sufficient for the moment recursion in Theorem~\ref{thm:Stein.recursion.disjoint.minors}.

Define the degree-lowering operator $\mathcal{R}_{\alpha,\Sigma}$ on polynomials $P:\mathcal{S}^p\to\R$ by
\begin{equation}\label{eq:lowering.operator}
(\mathcal{R}_{\alpha,\Sigma}P)(S) \leqdef \alpha \, \tr\{\Sigma\nabla P(S)\}  +  2 \, \tr\{S\nabla\Sigma\nabla P(S)\}, \qquad S\in \mathcal{S}^p.
\end{equation}
If $P\in\mathcal{H}_m$ with $m\geq 1$, then $\mathcal{R}_{\alpha,\Sigma}P\in\mathcal{H}_{m-1}$.

\begin{theorem}[Stein recursion for products of disjoint principal minors]\label{thm:Stein.recursion.disjoint.minors}
Let $d\in\N$, let $p_1,\ldots,p_d\in\N$, set $p\leqdef p_1 + \cdots + p_d$, and let $\mathfrak{W}=(\mathfrak{W}_{rs})_{1 \leq r,s \leq d}\sim \mathcal{W}_p(\alpha,\Sigma)$, where $\alpha\in (p-1,\infty)$ and $\Sigma\in \mathcal{S}_{++}^p$. For every $\bb{n}=(n_1,\ldots,n_d)\in\N_0^d$, let
\[
P_{\bb{n}}(S) = \prod_{r=1}^d |S_{rr}|^{n_r}, \qquad m_{\bb{n}} = \sum_{r=1}^d n_r p_r.
\]
If $m_{\bb{n}}=0$, then $\EE[P_{\bb{n}}(\mathfrak{W})]=1$. If $m_{\bb{n}}\geq 1$, define recursively
\[
P_{\bb{n}}^{(0)} \leqdef P_{\bb{n}}, \qquad P_{\bb{n}}^{(r + 1)} \leqdef \frac{1}{m_{\bb{n}}-r}\mathcal{R}_{\alpha,\Sigma}P_{\bb{n}}^{(r)}, \qquad r\in\{0,\ldots,m_{\bb{n}}-1\}.
\]
Then $\smash{P_{\bb{n}}^{(r)}\in\mathcal{H}_{m_{\bb{n}}-r}}$ for every $r\in\{0,\ldots,m_{\bb{n}}\}$, the polynomial $\smash{P_{\bb{n}}^{(m_{\bb{n}})}}$ is constant, and
\begin{equation}\label{eq:recursive.moment.formula}
\EE\left[\prod_{r=1}^d |\mathfrak{W}_{rr}|^{n_r}\right] = \EE[P_{\bb{n}}^{(0)}(\mathfrak{W})] = \EE[P_{\bb{n}}^{(1)}(\mathfrak{W})] = \cdots = \EE[P_{\bb{n}}^{(m_{\bb{n}})}(\mathfrak{W})] = P_{\bb{n}}^{(m_{\bb{n}})}.
\end{equation}
Equivalently,
\begin{equation}\label{eq:closed.form.recursion}
\EE\left[\prod_{r=1}^d |\mathfrak{W}_{rr}|^{n_r}\right] = \frac{1}{m_{\bb{n}}!}\left(\mathcal{R}_{\alpha,\Sigma}^{m_{\bb{n}}}P_{\bb{n}}\right)(S),
\end{equation}
for any $S\in \mathcal{S}^p$, because $\mathcal{R}_{\alpha,\Sigma}^{m_{\bb{n}}}P_{\bb{n}}$ is a constant polynomial.
\end{theorem}

\begin{proof}
The case $m_{\bb{n}}=0$ is immediate, because then $n_1=\cdots=n_d=0$ and $P_{\bb{n}}\equiv 1$. Assume from now on that $m_{\bb{n}}\geq 1$.

We first prove Euler's homogeneous-function identity (see \eqref{eq:Euler.identity.symmetric.matrices} below) in the present symmetric-matrix notation. Let $P\in\mathcal{H}_m$ with $m\geq 1$. By definition of homogeneity, $P(tS)=t^mP(S)$ for $t>0$ and $S\in \mathcal{S}^p$. Differentiating both sides with respect to $t$ and then setting $t=1$ gives
\begin{equation}\label{eq:Euler.first.side}
\left.\frac{\rd}{\rd t}P(tS)\right|_{t=1} = mP(S).
\end{equation}
On the other hand, the independent coordinates of $S\in\mathcal{S}^p$ are the diagonal entries $S_{ii}$ and the off-diagonal entries $S_{ij}$ with $i<j$. Therefore, by the ordinary chain rule in these coordinates,
\begin{equation}\label{eq:Euler.chain.rule}
\left.\frac{\rd}{\rd t}P(tS)\right|_{t=1} = \sum_{i=1}^p S_{ii}\frac{\partial P}{\partial S_{ii}}(S)  +  \sum_{1\leq i<j\leq p} S_{ij}\frac{\partial P}{\partial S_{ij}}(S).
\end{equation}
We now identify the right-hand side of \eqref{eq:Euler.chain.rule} with $\tr\{S\nabla P(S)\}$. By the definition of the symmetric gradient, $\nabla_{ii}P(S)=\partial P/\partial S_{ii}(S)$ and $\nabla_{ij}P(S)=\nabla_{ji}P(S)=\frac{1}{2}\partial P/\partial S_{ij}(S)$ for $i<j$. Since $S_{ij}=S_{ji}$, it follows that
\[
\begin{aligned}
\tr\{S\nabla P(S)\}
&= \sum_{i=1}^p S_{ii}\nabla_{ii}P(S)  +  \sum_{1\leq i<j\leq p}\{S_{ij}\nabla_{ji}P(S)  +  S_{ji}\nabla_{ij}P(S)\} \\
&= \sum_{i=1}^p S_{ii}\frac{\partial P}{\partial S_{ii}}(S)  +  \sum_{1\leq i<j\leq p} S_{ij}\frac{\partial P}{\partial S_{ij}}(S).
\end{aligned}
\]
Combining this identity with \eqref{eq:Euler.first.side} and \eqref{eq:Euler.chain.rule}, we obtain
\begin{equation}\label{eq:Euler.identity.symmetric.matrices}
\tr\{S\nabla P(S)\} = m P(S), \qquad S\in\mathcal{S}^p.
\end{equation}

We next check that $P_{\bb{n}}$ is homogeneous of degree $m_{\bb{n}}$. For $t>0$,
\[
P_{\bb{n}}(tS) = \prod_{r=1}^d |(tS)_{rr}|^{n_r} = \prod_{r=1}^d |tS_{rr}|^{n_r} = \prod_{r=1}^d t^{n_r p_r} |S_{rr}|^{n_r} = t^{m_{\bb{n}}} P_{\bb{n}}(S),
\]
so $P_{\bb{n}}\in\mathcal{H}_{m_{\bb{n}}}$.

Now let $P\in\mathcal{H}_m$ with $m\geq 1$. Since differentiation lowers total degree by one, $\tr\{\Sigma\nabla P(S)\}$ is homogeneous of degree $m-1$. Similarly, $\nabla\Sigma\nabla P(S)$ involves second derivatives of $P$, and therefore has degree $m-2$ when $m\geq 2$, while multiplication by $S$ raises the degree by one; if $m=1$, the second-derivative term is zero. Consequently, $\tr\{S\nabla\Sigma\nabla P(S)\}$ is homogeneous of degree $m-1$, and hence $\mathcal{R}_{\alpha,\Sigma}P\in\mathcal{H}_{m-1}$.

Applying the Wishart Stein identity \eqref{eq:Wishart.Stein.identity.polynomial} with this polynomial $P$ gives
\[
0 = 2\alpha \, \EE[\tr\{\Sigma\nabla P(\mathfrak{W})\}] - 2 \, \EE[\tr\{\mathfrak{W}\nabla P(\mathfrak{W})\}]  +  4 \, \EE[\tr\{\mathfrak{W}\nabla\Sigma\nabla P(\mathfrak{W})\}].
\]
By Euler's homogeneous-function identity \eqref{eq:Euler.identity.symmetric.matrices}, evaluated at $S=\mathfrak{W}$, we have $\tr\{\mathfrak{W}\nabla P(\mathfrak{W})\}=mP(\mathfrak{W})$. Substitution yields
\[
0 = 2\alpha \, \EE[\tr\{\Sigma\nabla P(\mathfrak{W})\}] - 2m \, \EE[P(\mathfrak{W})]  +  4 \, \EE[\tr\{\mathfrak{W}\nabla\Sigma\nabla P(\mathfrak{W})\}].
\]
Dividing by $2$ and rearranging gives the basic one-step recursion
\begin{equation}\label{eq:one.step.recursion}
\EE[P(\mathfrak{W})] = \frac{1}{m}\EE[(\mathcal{R}_{\alpha,\Sigma}P)(\mathfrak{W})].
\end{equation}

We now apply \eqref{eq:one.step.recursion} successively. By construction, $\smash{P_{\bb{n}}^{(0)}\in\mathcal{H}_{m_{\bb{n}}}}$. Since $\mathcal{R}_{\alpha,\Sigma}$ lowers degree by one, an induction gives $\smash{P_{\bb{n}}^{(r)}\in\mathcal{H}_{m_{\bb{n}}-r}}$ for every $r\in\{0,\ldots,m_{\bb{n}}\}$. For $r\in\{0,\ldots,m_{\bb{n}}-1\}$, applying \eqref{eq:one.step.recursion} to $\smash{P=P_{\bb{n}}^{(r)}}$, whose degree is $m_{\bb{n}}-r$, gives
\[
\EE[P_{\bb{n}}^{(r)}(\mathfrak{W})] = \frac{1}{m_{\bb{n}}-r}\EE[(\mathcal{R}_{\alpha,\Sigma}P_{\bb{n}}^{(r)})(\mathfrak{W})] = \EE[P_{\bb{n}}^{(r + 1)}(\mathfrak{W})].
\]
Iterating this equality proves \eqref{eq:recursive.moment.formula}. Finally, after $m_{\bb{n}}$ applications of $\mathcal{R}_{\alpha,\Sigma}$, the polynomial has degree zero, so $\smash{P_{\bb{n}}^{(m_{\bb{n}})}}$ is constant and its expectation is the constant itself. The closed expression \eqref{eq:closed.form.recursion} follows immediately.
\end{proof}

\begin{remark}\label{rem:algorithm.polynomial.space}
The recursion in Theorem~\ref{thm:Stein.recursion.disjoint.minors} is finite because $P_{\bb{n}}$ is a polynomial. The intermediate polynomials $\smash{P_{\bb{n}}^{(r)}}$ need not remain products of principal determinants; in general, they are polynomials in the entries of the full block matrix $S$. Thus the natural closed state space for the recursion is the space of homogeneous polynomials, not only the smaller family $\{\prod_{r=1}^d |S_{rr}|^{n_r}:\bb{n}\in\N_0^d\}$.
\end{remark}

\begin{remark}\label{rem:integer.powers}
The theorem is stated for nonnegative integer powers. For arbitrary real powers, the same Stein identity can still be applied under suitable integrability assumptions, but the derivative of $\prod_{r=1}^d |S_{rr}|^{\nu_r}$ introduces inverse block matrices and the degree-lowering polynomial recursion above no longer terminates in finitely many steps.
\end{remark}

\section{Numerical validation}\label{sec:validation}

We validate Theorem~\ref{thm:Stein.recursion.disjoint.minors} using randomly generated parameter configurations. Each configuration is determined by the Wishart parameters $(\alpha,\Sigma)$, the power vector $\bb{n}=(n_1,\ldots,n_d)$, and the block-size vector $\bb{p}=(p_1,\ldots,p_d)$, where $p_1 + \cdots + p_d=p$. The target quantity is
\[
M(\alpha,\Sigma,\bb{n},\bb{p}) \leqdef \EE\left[\prod_{r=1}^d |\mathfrak{W}_{rr}|^{n_r}\right], \qquad \mathfrak{W}\sim\mathcal{W}_p(\alpha,\Sigma).
\]

For each configuration, the moment is computed in two ways. The first is the exact recursive method of Theorem~\ref{thm:Stein.recursion.disjoint.minors}, which starts from $P_{\bb{n}}$, applies the degree-lowering operator $\mathcal{R}_{\alpha,\Sigma}$ exactly $m_{\bb{n}}=\sum_{r=1}^d n_rp_r$ times, and returns the final constant polynomial. The second is a direct Monte Carlo estimate based on independent Wishart samples. Monte Carlo estimates are computed using $10^3$, $10^5$, and $10^7$ replications to examine their convergence toward the exact recursive value as the sample size increases, and determinants are evaluated on a logarithmic scale for numerical stability.

The tables report, for each random configuration, the block-size vector $\bb{p}$, the power vector $\bb{n}$, the shape parameter $\alpha$, the condition number of $\Sigma$, the recursive value, the three Monte Carlo estimates, and their corresponding relative errors. For $N\in\{10^3,10^5,10^7\}$, the relative error is defined by
\begin{equation}\label{eq:RE}
\mathrm{RE}(N) \leqdef \frac{|M_{\mathrm{rec}}-\widehat M_{\mathrm{MC},N}|}{\max\{|M_{\mathrm{rec}}|,\e_{\mathrm{mach}}\}},
\end{equation}
where $\e_{\mathrm{mach}}$ denotes machine precision.

The vectors $\bb{p}$ and $\bb{n}$ are generated by rejection sampling. At each proposal step, the components of $\bb{p}$ are sampled independently and uniformly from $\{2,3,4\}$, while those of $\bb{n}$ are sampled independently and uniformly from $\{1,2,3,4\}$. The proposed pair $(\bb{p},\bb{n})$ is accepted only if $m_{\bb{n}}\leq 10$ when $d=2$, and only if $m_{\bb{n}}\leq 12$ when $d=3$ or $d=4$, to keep the computation time reasonable. Once $\smash{p=\sum_{r=1}^d p_r}$ is determined, the shape parameter is generated as $\alpha=p-1+U$, where $U$ is uniformly distributed on $(1,60)$. To generate $\Sigma$, recall that $\kappa(\Sigma)\leqdef \lambda_{\max}(\Sigma)/\lambda_{\min}(\Sigma)$ denotes its spectral condition number. A target value $K$ for $\kappa(\Sigma)$ is generated by sampling $\log K$ uniformly on $[\log(1.25),\log(50)]$. A random orthogonal matrix $Q$ is obtained from the QR decomposition of a matrix with independent standard normal entries. Before rescaling, the smallest and largest eigenvalues are set equal to $1$ and $K$, respectively, while the remaining eigenvalues are sampled independently on the logarithmic scale between $1$ and $K$. The complete collection of eigenvalues is then randomly permuted and multiplied by a common factor chosen so that the logarithm of the geometric mean of the eigenvalues is uniformly distributed on $[\log(0.5),\log(2)]$. Finally, $\Sigma$ is defined by $\Sigma=Q\Lambda Q^\top$, where $\Lambda$ is the diagonal matrix containing the resulting eigenvalues.

Table~\ref{tab:wilks.validation.b2} gives the comparison for two disjoint principal minors. This case is especially useful as a benchmark, because an independent exact formula involving the Gaussian hypergeometric function of matrix argument is available when $d=2$, due to \citet{MR4862164}. For general real exponents, numerical evaluation of this hypergeometric function is commonly based on truncating its defining series and is therefore approximate. For the nonnegative integer powers considered here, however, its two upper parameters are negative integers, so the series terminates and no truncation error is introduced. We nevertheless omit this additional computation from Table~\ref{tab:wilks.validation.b2} in order to retain a uniform comparison between the recursive and Monte Carlo methods across all values of $d$. Separate numerical experiments (not reported here) confirm agreement to machine precision between the new recursive method and the formula of \citet[Theorem~3.1]{MR4862164} when $d=2$.

Tables~\ref{tab:wilks.validation.b3} and~\ref{tab:wilks.validation.b4} give the corresponding comparisons for three and four disjoint principal minors, respectively, for which no comparable closed formula is currently available. The agreement between the recursive and Monte Carlo values, together with the overall tendency of the relative errors to decrease as the number of Monte Carlo replications increases, supports the implementation of the recursion and the expected convergence of the Monte Carlo estimates.

\vspace{1mm}
\input{R_code/wilks_validation_b2_paper.tex}

\input{R_code/wilks_validation_b3_paper.tex}

\input{R_code/wilks_validation_b4_paper.tex}

\section*{Funding}
\addcontentsline{toc}{section}{Funding}

Robert E.\ Gaunt is funded by EPSRC grant EP/Y008650/1. Fr\'ed\'eric Ouimet is supported by the Natural Sciences and Engineering Research Council of Canada (NSERC) through Discovery Grant RGPIN-2026-04471 and Discovery Launch Supplement DGECR-2026-00449.

\section*{References}
\addcontentsline{toc}{section}{References}

\setlength{\bibsep}{0pt plus 0ex}

\bibliographystyle{plainnat}
\bibliography{bib}

\end{document}

%% file: R_code/wilks_validation_b2_paper.tex
\begin{table}[!ht]
\centering
\caption{Numerical validation of Wilks moments for $d = 2$ disjoint principal minors. The recursive column is the exact value obtained from Theorem~\ref{thm:Stein.recursion.disjoint.minors}; the Monte Carlo columns use $10^3$, $10^5$, and $10^7$ replications, respectively. Here $\mathrm{RE}$ denotes the relative error defined in \eqref{eq:RE}.}
\footnotesize
\label{tab:wilks.validation.b2}
\begingroup
\renewcommand{\arraystretch}{1}
\setlength{\tabcolsep}{2pt}
\begin{tabular}{crrrrrrrrrrr}
\hline
Case & $\bb p$ & $\bb n$ & $\alpha$ & $\kappa(\Sigma)$ & Recursive & MC($10^3$) & MC($10^5$) & MC($10^7$) & RE($10^3$) & RE($10^5$) & RE($10^7$)
\\
\hline
1 & 4,2 & 1,2 & 51.71 & 21.62 & 5.8323e+12 & 5.8600e+12 & 5.8320e+12 & 5.8318e+12 & 0.0047 & 5.8090e-05 & 8.7683e-05 \\
2 & 2,2 & 4,1 & 61.68 & 5.54 & 2.5107e+19 & 2.6257e+19 & 2.5268e+19 & 2.5106e+19 & 0.0458 & 0.0064 & 2.8303e-05 \\
3 & 4,3 & 1,1 & 15.62 & 1.85 & 2.5259e+07 & 2.6111e+07 & 2.5192e+07 & 2.5265e+07 & 0.0337 & 0.0027 & 2.3816e-04 \\
4 & 2,2 & 2,2 & 39.56 & 47.75 & 2.8244e+16 & 2.9893e+16 & 2.8526e+16 & 2.8224e+16 & 0.0584 & 0.0100 & 7.1199e-04 \\
5 & 4,3 & 1,1 & 62.42 & 2.27 & 3.9106e+10 & 3.9488e+10 & 3.9110e+10 & 3.9102e+10 & 0.0098 & 1.0813e-04 & 9.0291e-05 \\
6 & 3,4 & 2,1 & 62.07 & 15.53 & 2.2785e+16 & 2.3260e+16 & 2.2735e+16 & 2.2775e+16 & 0.0208 & 0.0022 & 4.5215e-04 \\
7 & 2,2 & 1,4 & 62.43 & 2.22 & 2.6083e+17 & 2.6869e+17 & 2.6140e+17 & 2.6067e+17 & 0.0301 & 0.0022 & 6.2915e-04 \\
8 & 2,4 & 1,2 & 56.25 & 4.75 & 2.6014e+19 & 2.7610e+19 & 2.5902e+19 & 2.6024e+19 & 0.0613 & 0.0043 & 3.6102e-04 \\
9 & 2,3 & 2,2 & 32.35 & 12.53 & 1.1761e+18 & 1.1933e+18 & 1.1819e+18 & 1.1760e+18 & 0.0145 & 0.0049 & 1.0415e-04 \\
10 & 3,3 & 1,2 & 8.31 & 2.66 & 2.0453e+06 & 1.9078e+06 & 2.0235e+06 & 2.0466e+06 & 0.0672 & 0.0106 & 6.6288e-04 \\
\hline
\end{tabular}
\endgroup
\end{table}

%% file: R_code/wilks_validation_b3_paper.tex
\begin{table}[!ht]
\centering
\caption{Numerical validation of Wilks moments for $d = 3$ disjoint principal minors. The recursive column is the exact value obtained from Theorem~\ref{thm:Stein.recursion.disjoint.minors}; the Monte Carlo columns use $10^3$, $10^5$, and $10^7$ replications, respectively. Here $\mathrm{RE}$ denotes the relative error defined in \eqref{eq:RE}.}
\footnotesize
\label{tab:wilks.validation.b3}
\begingroup
\renewcommand{\arraystretch}{1}
\setlength{\tabcolsep}{2pt}
\begin{tabular}{crrrrrrrrrrr}
\hline
Case & $\bb p$ & $\bb n$ & $\alpha$ & $\kappa(\Sigma)$ & Recursive & MC($10^3$) & MC($10^5$) & MC($10^7$) & RE($10^3$) & RE($10^5$) & RE($10^7$)
\\
\hline
1 & 2,2,2 & 1,1,1 & 13.35 & 1.80 & 7.5337e+06 & 7.1858e+06 & 7.5072e+06 & 7.5275e+06 & 0.0462 & 0.0035 & 8.1792e-04 \\
2 & 3,2,2 & 1,2,1 & 18.74 & 1.48 & 6.0604e+12 & 6.0622e+12 & 6.0677e+12 & 6.0582e+12 & 2.9444e-04 & 0.0012 & 3.5853e-04 \\
3 & 3,3,2 & 1,1,2 & 66.94 & 5.30 & 1.1176e+18 & 1.1169e+18 & 1.1132e+18 & 1.1174e+18 & 5.8431e-04 & 0.0039 & 1.5759e-04 \\
4 & 3,2,2 & 2,1,1 & 26.75 & 44.39 & 2.7288e+17 & 2.6060e+17 & 2.7220e+17 & 2.7287e+17 & 0.0450 & 0.0025 & 4.2188e-05 \\
5 & 2,3,2 & 1,2,1 & 33.92 & 2.99 & 1.3337e+18 & 1.3627e+18 & 1.3313e+18 & 1.3343e+18 & 0.0217 & 0.0018 & 3.9762e-04 \\
6 & 4,2,2 & 1,1,1 & 34.55 & 23.96 & 8.8707e+11 & 9.4813e+11 & 8.8790e+11 & 8.8757e+11 & 0.0688 & 9.3506e-04 & 5.6347e-04 \\
7 & 2,4,2 & 1,1,1 & 26.69 & 15.73 & 4.1555e+13 & 4.1577e+13 & 4.1524e+13 & 4.1559e+13 & 5.2667e-04 & 7.5492e-04 & 9.5061e-05 \\
8 & 2,3,2 & 2,1,2 & 44.04 & 10.82 & 1.0483e+17 & 1.0275e+17 & 1.0493e+17 & 1.0484e+17 & 0.0198 & 9.9133e-04 & 1.4182e-04 \\
9 & 2,2,2 & 3,1,2 & 24.76 & 9.97 & 1.3192e+15 & 1.1819e+15 & 1.3248e+15 & 1.3184e+15 & 0.1041 & 0.0043 & 5.7245e-04 \\
10 & 2,2,2 & 2,3,1 & 31.88 & 29.89 & 5.0062e+17 & 5.3702e+17 & 5.0094e+17 & 5.0094e+17 & 0.0727 & 6.4657e-04 & 6.5184e-04 \\
\hline
\end{tabular}
\endgroup
\end{table}

%% file: R_code/wilks_validation_b4_paper.tex
\begin{table}[!ht]
\centering
\caption{Numerical validation of Wilks moments for $d = 4$ disjoint principal minors. The recursive column is the exact value obtained from Theorem~\ref{thm:Stein.recursion.disjoint.minors}; the Monte Carlo columns use $10^3$, $10^5$, and $10^7$ replications, respectively. Here $\mathrm{RE}$ denotes the relative error defined in \eqref{eq:RE}.}
\footnotesize
\label{tab:wilks.validation.b4}
\begingroup
\renewcommand{\arraystretch}{1}
\setlength{\tabcolsep}{2pt}
\begin{tabular}{crrrrrrrrrrr}
\hline
Case & $\bb p$ & $\bb n$ & $\alpha$ & $\kappa(\Sigma)$ & Recursive & MC($10^3$) & MC($10^5$) & MC($10^7$) & RE($10^3$) & RE($10^5$) & RE($10^7$)
\\
\hline
1 & 2,3,2,2 & 1,1,1,1 & 27.28 & 4.75 & 7.2115e+12 & 7.3220e+12 & 7.1829e+12 & 7.2076e+12 & 0.0153 & 0.0040 & 5.4261e-04 \\
2 & 2,2,2,2 & 1,2,1,1 & 22.61 & 1.40 & 1.1902e+16 & 1.3067e+16 & 1.1904e+16 & 1.1907e+16 & 0.0979 & 2.1223e-04 & 4.3843e-04 \\
3 & 2,3,2,2 & 1,1,1,1 & 23.79 & 40.19 & 8.3047e+15 & 8.2041e+15 & 8.2796e+15 & 8.3054e+15 & 0.0121 & 0.0030 & 8.8461e-05 \\
4 & 2,2,2,2 & 1,1,2,1 & 23.65 & 1.31 & 1.0388e+14 & 1.0750e+14 & 1.0363e+14 & 1.0382e+14 & 0.0349 & 0.0024 & 6.1552e-04 \\
5 & 2,2,2,3 & 1,1,1,1 & 27.25 & 17.16 & 8.4051e+16 & 8.8546e+16 & 8.3994e+16 & 8.4075e+16 & 0.0535 & 6.7176e-04 & 2.8379e-04 \\
6 & 2,2,2,2 & 2,1,1,1 & 50.95 & 4.79 & 1.6460e+19 & 1.6856e+19 & 1.6493e+19 & 1.6449e+19 & 0.0241 & 0.0020 & 6.9551e-04 \\
7 & 3,2,2,2 & 1,1,1,1 & 34.49 & 3.45 & 2.4793e+12 & 2.5947e+12 & 2.4664e+12 & 2.4794e+12 & 0.0465 & 0.0052 & 5.8522e-05 \\
8 & 2,2,2,2 & 2,1,1,1 & 32.22 & 7.95 & 1.0414e+16 & 1.0240e+16 & 1.0454e+16 & 1.0414e+16 & 0.0167 & 0.0038 & 4.7824e-05 \\
9 & 2,2,2,2 & 2,1,1,1 & 18.80 & 7.41 & 9.6223e+14 & 1.0421e+15 & 9.7000e+14 & 9.6195e+14 & 0.0830 & 0.0081 & 2.9472e-04 \\
10 & 2,2,2,2 & 1,1,1,1 & 23.36 & 2.22 & 3.3027e+11 & 3.3213e+11 & 3.2953e+11 & 3.3010e+11 & 0.0056 & 0.0023 & 5.1926e-04 \\
\hline
\end{tabular}
\endgroup
\end{table}